\documentclass[12pt]{amsart}
\usepackage{amsmath}
\usepackage{geometry,amsfonts,amssymb,amsthm,txfonts,pxfonts,amscd} 
\def\struckint{\mathop{%
\def\mathpalette##1##2{\mathchoice{##1\displaystyle##2}%
  {##1\textstyle##2}{##1\scriptstyle##2}{##1\scriptscriptstyle##2}}%
\mathpalette
{\vbox\bgroup\baselineskip0pt\lineskiplimit-1000pt\lineskip-1000pt
\halign\bgroup\hfill$}
{##$\hfill\cr{\intop}\cr\diagup\cr\egroup\egroup}%
}\limits}
\newtheorem{theorem}{Theorem}
\newtheorem{lemma}[theorem]{Lemma}
\newtheorem{corollary}[theorem]{Corollary}

\theoremstyle{remark}

\newtheorem{remark}[theorem]{Remark}

\newtheorem{question}[theorem]{Question}

\begin{document}

\title{On extension of  coverings}

\author{Manfred Droste}
\address{Computer Science Department, University of Leipzig}
\email{droste@informatik.uni-leipzig.de}
\author{Igor Rivin}
\address{Department of Mathematics, Temple University, Philadelphia}
\email{rivin@temple.edu}
\thanks{Igor Rivin was supported by the National Science Foundation. He would also like to thank Dennis Sullivan and Nick Katz for enlightening discussions.}
\date{\today} 
\subjclass{57M12, 20B07}
\keywords{covering space, extension, infinite symmetric group, commutator, automorphism group}


\begin{abstract}
We address the question of when a covering of the boundary of a surface $F$ can be extended to a covering of the surface (equivalently: when is there a branched cover with a prescribed monodromy ).  If such an extension is possible, when can the total space be taken to be connected? When can the extension be taken to be regular? We give necessary and sufficient conditions for both finite and infinite covers (infinite covers are our main focus). In order to prove our results, we show group-theoretic results  of independent interests, such as the following extension (and simplification) of the theorem of Ore \cite{orecomm}: every element of the infinite symmetric group is the commutator of two elements which, together, act transitively.We also make some general remarks on writing elements of various groups as commutators.
\end{abstract}
\maketitle


\section*{Introduction}
\label{coversec}
This paper centers around the following question:
\begin{question}
\label{coverq}
Given a  surface $S$ with boundary $\partial S,$ and a covering map of $1$-manifolds 
$\pi: \widetilde{C} \rightarrow \partial S$ the fibers of which have constant cardinality. When does $\pi$ extend to a covering map $\Pi: \widetilde{S}\rightarrow S,$ where $\partial \widetilde{S} \simeq \widetilde{C}?.$ If such a cover exists, how do we construct it?
\end{question}
An alternative (but equivalent) version is the following:

\begin{question}
\label{coverq2}
Given a surface $S$ with a finite collection of distinguished points $\{p_1, \dotsc, p_n\}$ and an assignment of elements $\pi_i \in S_\Omega$ to each $p_i$, when is there a branched covering map
of $S$ of index $|\Omega|,$ such that the monodromy of the branched cover around the branch points is  given by $\pi_1, \dots, \pi_n?$  If such a cover exists, how do we construct it?
\end{question}

Questions \ref{coverq} and \ref{coverq2} (together with some history) are discussed in Section \ref{finitecov}.  The results are very clean in the case of surfaces of positive genus, and considerably more complicated for the disk. In Section \ref{infinitecov} we will be studying infinite covers. In this case, we get an even cleaner result for surfaces of positive genus, which depends on our Theorem \ref{orethm}, which states that every element of the infinite symmetric group $S_{\Omega}$ is the commutator of two elements which together act \emph{transitively}on $\Omega.$  We extend our results to non-orientable surfaces (of positive genus). These results depend on Theorems \ref{powers} and \ref{dglemma}, which extend results of \cite{Drosteuniv} to the transitive context.  Theorems \ref{orethm} and \ref{powers} appear to be of independent group-theoretic interest. We also obtain some partial results in the subtle case where $S$ is  a punctured sphere -- in this case the problem is overdetermined, hence the difficulty. A self-contained proof of Theorem \ref{orethm} is contained in Sections \ref{transitiveore},\ref{infinite}, and \ref{finitecycle}. In Section \ref{regcover} we comment on the question of when the covering space can be further chosen to be regular. In Section \ref{dig} we give a partial list of groups which have the surprising property that every element is a commutator -- this has far-reaching applications for a number of extension problems.

\section{Finite Covers}
\label{finitecov}
 First, we note that to every degree $n$ covering map $\sigma: X\rightarrow Y$ we can associate a permutation representation $\Sigma: \pi_1(Y) \rightarrow S_n.$ Further, two coverings $\sigma_1$ and $\sigma_2$ are equivalent if and only if the associated representations $\Sigma_1$ and $\Sigma_2$ are conjugate (see, for example, \cite{hatcheralgtop}[Chapter 1] for the details). This means that the boundary covering map $\pi$ is represented by  a collection of $k=|\pi_0(\partial S)|$ conjugacy classes $\Gamma_1, \dotsc, \Gamma_k$  in $S_n.,$ each of which is the conjugacy class of the image of the generator of the fundamental group of the corresponding component under the associated permutation representation. Theorem \ref{sullthm} below is essentially due to A. Hurwitz \cite{hurwitzriemann} and D. Husemoller \cite{husemollercovers}
\begin{theorem}
\label{sullthm}
A covering $\pi:\widetilde{C}\rightarrow C=\partial S$ extends to a covering of the surface $S$ if and only if one of the following conditions holds:
\begin{enumerate}
\item $S$ is a planar surface (that is, the genus of $S$ is zero and there exists a collection $\{\sigma_i\}_{i=1}^k$ of elements of the symmetric group $S_n$ with $\sigma_i \in \Gamma_i$ such that 
$\sigma_1 \sigma_2 \dots \sigma_k = e,$ where $e\in S_n$ is the identity.
\item $S$ is not a planar surface, and the sum of the parities of $\Gamma_1, \dotsc, \Gamma_k$ vanishes.
\end{enumerate}
\end{theorem}
\begin{proof}
In the planar case, the fundamental group of $S$ is freely generated by the generators $\gamma_, \dotsc, \gamma_{k-1}$ of the  fundamental groups of (any) $k-1$ of the boundary components.  All $k$ boundary components satisfy $\gamma_1 \dotsc \gamma_k = e,$ whence the result in this case.

In the nonplanar case, let us first consider the case where $k=1.$ The generator $\gamma$ of the single boundary component is then a product of $g$ commutators (where $g$ is the genus of the surface, and so $\Sigma(\gamma)$ is in the commutator subgroup of $S_n,$ which is the alternating group $A_n,$ so the class $\Sigma(\gamma)$ has to be even. On the other hand, it is a result of O. Ore \cite{orecomm} that any even permutation is a commutatior\footnote{Ore showed that every even permutation is a commutator of two permutation. At roughly the same time, N. Ito showed in \cite{itoalt} that every even permutation is a commutator of two \emph{even} permutations.} $\alpha \beta \alpha^{-1} \beta^{-1}$ and thus sending some pair of handle generators to $\alpha$ and $\beta$ respectively and the other generators of $\pi_1(S)$ to $e$ defines the requisite homorphism of $\pi_1(S)$ to $S_n.$

If $k>1,$ the surface $S$ is a connected sum of a surface of genus $g>0$ (by assumption) and a planar surface with $k$ boundary components.  Let $\gamma$ be the ``connected summing'' circle. By the planar case, there is no obstruction to defining $\Pi$ on the planar side (since $\gamma$ is not part of the original data). However, $\Sigma(\gamma)$ will be the inverse of the product of elements $\sigma_i \in \Gamma_i$ and so its parity will be the sum of the parities of $\Gamma_i.$ To extend the cover to the non-planar side of the connected sum, it is necessary and sufficient for this sum to be even.
\end{proof}

Some remarks are in order. The first one concerns the planar case of Theorem \ref{sullthm}. It is not immediately obvious how one might be able to figure out whether given some conjugacy classes in the symmetric group, there are representatives of these classes which multiply out to the identity. Luckily, there is the following result of Frobenius (see \cite{serregalois}[p. 69])
\begin{theorem}
\label{frob}
Let $C_1, \dotsc, C_k$ be conjugacy classes in a finite group $G.$ The number $n$ of solutions to the equation $g_1 g_2 \dots g_k = e,$ where $g_i\in C_i$ is given by 
\[
n = \dfrac{1}{|G|} |C_1| \dots |C_k| \sum_{\chi} \dfrac{\chi(x_1) \dots \chi(x_k)}{\chi(1)^{k-2}},
\]
where $x_k \in C_k$ and the sum is over all the complex irreducible characters of $G.$
\end{theorem}
The planar case is considered in \cite{Kulkarnicovers}, see also \cite{Pervova1,Pervova2,Pervova3}
The more elaborate questions of \emph{how many} covers exist with given data enumeration questions for covers are considered in a number of papers by A.~Mednykh -- see \cite{Mednykhenum} and references therein, and also by A. Eskin and A. Okounkov (see \cite{EskinOkounkov})-- the last paper studies asymptotic growth of the number of such coverings.

The second remark is on Ore's result that every element of $A_n$ is even. This result was strengthened by A.~Gleason in \cite{husemollercovers} and later byE. Bertram in \cite{bertramcycles} . 
The significance of this to coverings two-fold: Firstly, an $n$-cycle is transitive, and so we have the existence of a \emph{connected} extension. Second is that we have a very simple way of constructing a covering of a surface with one boundary component with specified cycle structure of the covering of the component, as follows:

First, the proof of Theorem \ref{sullthm} shows that the construction reduces to the case where $g=1,$ so that we are constructing a covering of a torus with a single perforation.

Suppose now that the permutation can be written as $\sigma\tau\sigma^{-1}\tau^{-1},$ where $\sigma$ is an $n$-cycle. This means that the "standard" generators of the punctured torus group go to $\sigma$ and $\tau,$ respectively. To construct the cover, then, take the standard square fundamental domain $D$ for the torus (the puncture is at the vertices of the square), then arrange $n$ of these fundamental domains in a row, and then a strip, by gluing the rightmost edge to the leftmost edge. Then, for each $i,$ the upper edge of the $i$-th domain from the left ($D_i$) is glued to the lower edge of $D_{\tau(i)}.$

\section{A generalization and a digression}
\label{dig}
The argument in the preceding section uses the following observation:
\begin{theorem}
\label{commutatorthm}
Let $G$ be group where every element $g\in G$ can be written as a commutator $g=hkh^{-1} k^{-1}.$ Then, any homomorphism from the fundamental group of the boundary of an orientable surface $S$ into $G$ can be extended to a hommorphism of $\pi_1(S)$ into $G.$
\end{theorem}

What is interesting,is how many groups $G$  satisfy the hypothesis of Theorem \ref{commutatorthm}. These include:
\begin{enumerate}
\item Finite and infinite alternating groups (see section \ref{finitecov}).
\item The infinite symmetric group (see section \ref{infinitecov} for application, and section \ref{transitiveore} for a strengthening.)
\item (Conjecturally) all finite non-abelian simple groups (see \cite{shalevcomm} for a comprehensive survey of what is known -- the conjecture is known to be true for almost all the groups in the census).
\item The group of homeomorphisms of $S^1$ (see \cite{eisenmann})
\item All semi-simple connected algebraic groups (see \cite{reesemisimp})
\item All semi-simple complex Lie groups (see \cite{leesemisimp})
\item The groups of order automorphisms of rationals and the reals. (see \cite{higmanperm,hollandlat,wielandtinf})
\item The groups of automorphisms of the random colored graphs\footnote{These are a generalization of ``the random graph''  $R$ of Erdos and Renyi, which has the property that any countable random  graph (obtained by selecting edges of the complete graph on $\omega$ vertices, each with probability $1/2$) is almost surely isomorphic to $R$}(\cite{trussgen}) and of the universal homogeneous partial ordering (\cite{glassmcclearyrubin,kusketruss})
\item The group of homeomorphisms of the Cantor set (follows from the result of \cite{kechrisrosendal} (see \cite{rokhlinprop}  for a nice survey and additional results, and see \cite{parrysullivan} for topological and dynamical applications) to the effect that this group has the Strong Rokhlin Property).
\end{enumerate}

These facts, together with Theorem \ref{commutatorthm} have diverse applications (to extending covers, to extending flat bundles, to constructing Seifert fibrations, and doubtlessly many other applications we are not aware of).

If a group is perfect and has the \emph{Bergman property} (that is, that the Cayley graph with respect to any generating set has finite diameter, see \cite{bergmanprop}), then a covering of the boundary extends to a covering of the surface of a sufficiently high genus.

\subsection{Infinite covers}
\label{infinitecov}
The real point of O.~Ore's paper \cite{orecomm} is to show that \emph{every} element of the \emph{infinite} symmetric group is a commutator. This, together with the argument used to show Theorem \ref{sullthm} shows
\begin{theorem}
\label{sullthminf}
Let $S$ be a surface of genus $g>0,$ with at least one boundary component. Then \emph{every} covering of $\partial S$ of infinite index extends to a covering of $S.$
\end{theorem}

However, Ore's theorem does \emph{not} show the existence of a connected covering space. To show that, we need the stronger Theorem \ref{orethm}, which strengthens Ore's theorem to show that every element of $S_\infty$ can be taken to be the commutator of elements $g$ and $h$ which together act transitively on the set of sheets. 
\begin{theorem}
\label{orecorr}
The covering whose existence is shown by Theorem \ref{sullthminf} can be taken to be connected.
\end{theorem}
The proof of Theorem \ref{orethm} is completely explicit, and so a covering postulated by theorem \ref{orecorr} can be constructed explicitely (whatever that means for infinite covers). In addition, the proof of Theorem \ref{orethm} easily implies that:
\begin{theorem}
\label{extcount}
There is an \emph{uncountable} number of extensions of a given infinite covering.
\end{theorem}

\subsection{Non-orientable surfaces}
Our results extend \emph{verbatim} to coverings of non-orientable surfaces as follows:
\begin{theorem}
\label{sullthminfn}
Let $S$ be a non-orientable of genus $g>0,$ with at least one boundary component. Then \emph{every} covering of $\partial S$ of infinite index extends to a \emph{connected} covering of $S.$
\end{theorem}
\begin{proof}
Let $p$ be the number of boundary components. For such a covering to extend, the product of the monodromies of the boundary components must be a product of at most $p+1$ squares $\beta_1, \dotsc, \beta_{p+1}$, and the monodromies together with $\beta_1, \dotsc, \beta_{p+1}$ must generate a transitive subgroup. This, however, is the content of Theorem \ref{powers} (with $k=m=1$) and Theorem \ref{dglemma} (with $k=2$).
\end{proof}
\begin{proof}[A simple quasi-Proof]
\label{threesquares}
A simple proof of Theorem \ref{sullthminfn} for surfaces of genus greater than $1$ comes from the following fact (see \cite{lyndonnewman}): for  any two elements $x, y$ in a group,
\[
[x, y] = x y x^{-1} y^{-1} = (xy)^ 2(y^{-1} x^{-1} y)^2 (y^{-1})^2.
\]
Since the subgroup generated by $xy$ and $y^{-1}$ includes all of $\langle x, y\rangle,$ it follows that if the latter acts transitively, so does the former. We can now use Theorem \ref{orethm} to finish the argument. The main result of \cite{lyndonnewman} is that not every commutator can be expressed with \emph{fewer} than three squares, so this method really does not extend to genus $1.$
\end{proof}

\subsection{Infinite planar covers}
Our results for infinite planar covers are a bit less satisfying. The question is obviously trivial when the surface is a disk, and when the surface is a punctured disk (a cylinder), it is clear that the only connected infinite cover is the infinite cyclic cover, and so the only monodromy which extends is that of an infinite transitive cycle (the same on both boundary components). We now skip over the thrice-punctured sphere, to state a somewhat satisfying result for spheres with four or more punctures. That follows from the following result of the first-named author (\cite{Drosteprod}):
\begin{theorem}
\label{droste83}
Let $P_1, \dots, P_k, \quad k>3$ be conjugacy classes in $S_\infty.,$ each of which has infinite support (that is, moves an infinite number elements). Then, there are representatives $p_1 \in P_1, \dotsc, p_k \in P_k,$ such that $p_1 \centerdot p_2 \centerdot \dots p_k = e.$
\end{theorem}
\begin{corollary}
\label{drostecor}
Given a covering of the disjoint union of  $k$ circles  with monodromy given by $P_1, \dots, P_k$ satisfying the conditions of Theorem \ref{droste83}, there is a covering of the $k$-punctured sphere with that monodromy.
\end{corollary}
Notice that the statement of Corollary \ref{drostecor} makes no comment on the connectedness of the resulting cover. It is fairly clear that some in-finiteness hypothesis is necessary (since, eg, if \emph{all} the conjugacy classes have finite support, we are very close to the finite situation, where no uniformly positive answer exists, but it seems plausible that the hypothesis that \emph{all} of $P_1, \dots, P_k$ have finite support is too strong.

Finally, we are in the case $k=3.$ Here, we can use Theorem A of \cite{Drosteprod}, which states the following:
\begin{theorem}
\label{thmA}
Let $S, P_1, P_2$ be conjugacy classes which all have infinite support and such that $P_1$ and $P_2$ contain at least one infinite cycle in their decomposition. Then, there are representatives $s, p_1, p_2$ in $S, P_1, P_2$ respectively, such that $sp_1 p_2 = e.$
\end{theorem}
We have the immediate corollary:
\begin{corollary}
\label{drostecor2}
Given a covering of the disjoint union of $3$ circles satisfying the conditions of Theorem \ref{thmA}, such a covering extends to a covering of the thrice-punctured sphere.
\end{corollary}
Again, we make no statement regarding connectedness. In the case of $k=3$ we also have the following negative result:
\begin{theorem}
\label{bertthm}
There is no covering of the thrice punctured sphere which is infinite cyclic over two of the components and such that the third component is covered finitely, with an odd number of components of even index.
\end{theorem}
\begin{proof}
This is an immediate corollary of a result of E. Bertram (see \cite{Bertraminf}) which states that an odd permutation of finite support is \emph{not} a product of two infinite cycles.
\end{proof}

\subsection{Regular coverings}
\label{regcover}
D. Futer asked whether the results of the previous section had analogues when the covering given by $\Pi$ was additionally required to be \emph{regular}. This seems to be a hard question in general. For example, in the case where $S$ is a planar surface with $k$ boundary components, we have the following result:
\begin{theorem}
\label{regplanar}
In order for a covering of the boundary of $S$ to extend to a \emph{regular} covering of $S,$ it is necessary and sufficient that, in addition to the requirements of Theorem \ref{sullthm} (part 1), there must be a subgroup $G < S_n,$ with $|G| = n$ and  $G$ is generated by $\gamma_1, \dots, \gamma_k,$ where $\gamma_i \in \Gamma_i,$ for $i=1, \dots, k.$
\end{theorem}

As far as the authors know, there is no particularly efficient way of deciding whether the condition of Theorem \ref{regplanar} is satisfied.

Here is a more satisfactory (if not very positive) result:
\begin{theorem}
\label{regq}
Let $S$ be a surface with one boundary component. There does not exist a nontrivial covering of finite index  $\Pi: \widetilde{S}\rightarrow S$ where $\widetilde{S}$ also has one boundary component.
\end{theorem}

\begin{proof}
Let the degree of the covering be $n.$ If $\widetilde{S}$ has one boundary component, the generator of the $\pi_1(\partial S)$ gives rise to the cyclic group $\mathbb{Z}/n\mathbb{Z},$ which is a subgroup of the deck group of $\Pi.$. Since the deck group has order $n$ (by regularity), the covering is cyclic (so that the deck group is, in fact, $\mathbb{Z}/n \mathbb{Z}.$). A cyclic group is abelian, and since the generator of $\pi_1(\partial S)$ is a product of commutators, it is killed by the map $\Sigma.$ But this contradicts the statement of the first sentence of this proof (that this same element generates the entire deck group).
\end{proof}

The proof also shows the following:
\begin{theorem}
\label{regq2}
Let $\Pi: \widetilde{S}\rightarrow S,$ where $S$ has a single boundary component, be an \emph{abelian} regular covering of degree $n.$ Then $\widetilde{S}$ has $n$ boundary components.
\end{theorem}

We can combine our results in the following omnibus theorem:
\begin{theorem}
\label{regq3}
Let $S$ be a surface, whose boundary has $k$ connected components. Let the conjugacy classes of the $k$ coverings be $\Gamma_1, \dotsc, \Gamma_k.$  In order for a constant cardinality covering of $\partial S$ to extend to a regular covering of $S$ it is necessary and sufficient that there be elements $\gamma_1\in \Gamma_1, \dotsc, \gamma_k \in \Gamma_k$ such that $\gamma_1, \dotsc, \gamma_k$ generate an order $n$ subgroup of $S_n$ and $\gamma_1\dots \gamma_k = e.$
\end{theorem}
Theorem \ref{regq2} can be extended to infinite degree coverings, in the spirit of Theorem \ref{sullthminf}, as follows:

\begin{theorem}
\label{regq4}
Let $\Pi:\widetilde{S}\rightarrow S,$ where $S$ has a single boundary component, be an abelian regular covering of infinite degree. Then, every connected component of $\Pi^{-1}(\partial S)$ covers $\partial S$ exactly once.
\end{theorem}

\section{A transitive Ore theorem}
\label{transitiveore}
In his short but important paper \cite{orecomm}, Oystein Ore showed two closely related results, both of which have spawned active areas of research. The first was that every element of the alternating group $A_n < S_n $ is a commutator in $S_n.$ (Ore stated, but did not give a proof that the result can be sharpened to give every even permutation as a commutator of even permutations). This result was reproved by a number of people, but also gave rise to Ore's conjecture that every element of a finite (non-abelian) simple group is a commutator and the stronger Thompson's conjecture to the effect that every finite non-abelian simple group is the square of a conjugacy class. 

Ore's second result (and one whose proof occupies ninety percent of his paper) is that every element $\sigma$ of the group $S_\infty$ of permutations of a countable set $\Omega$ is a commutator $[g, h]$. Ore's proof is quite complicated, involving a case-by-case analysis of the possible cycle decompositions of $\sigma;$ different cycle decompositions give quite different sorts of $g$ and $h.$ While Ore's result has been sharpened and generalized by a number of people (including one of the authors of this paper). A relatively simple proof has recently been given by George Bergman in the celebrated paper \cite{bergmanprop}. Here we prove the following strengthening:

\begin{theorem}
\label{orethm}
Every element $\sigma \in S_\infty$ can be given as a commutator $[g, h],$ where $g$ and $h$ generate a subgroup $K < S_\infty$ acting transitively on $\Omega.$
\end{theorem}

The property of transitivity is important for our application (extending covering maps of surfaces); nevertheless the proof of Theorem \ref{orethm} is simpler than Ore's original proof, although it is based on the same basic idea: to represent the set $\Omega$ (which one would naturally tend to identify with $\mathbb{Z}$) with $\mathbb{Z}^2.$ This idea has also been exploited to considerable advantage in \cite{Drosteuniv}.

The proof involves a case analysis, but this time there are only two cases, dealt with in the two sections below. First, a bit of notation. Given a point $p=(i, j) \in \mathbb{Z}^2,$ we set $x(p) = i,\ y(p) = j.$ An \emph{infinite cycle} of a permutation $\sigma$ is a bi-infinite sequence 
$(\dots, a_{-n}, \dots a_0, \dots a_n, \dots), $ such that $\sigma(a_i) = a_{i+1}.$ We let $\overline{\sigma}$ be the function associating to each $n \in \mathbb{N} \cup \infty$ the number of cycles (or orbits) of size $n$ of $\sigma.$ 
We will need the following simple well-known lemma.
\begin{lemma}
\label{conjugacy1}
Let $\sigma, \tau \in S_\infty..$ Then $\sigma$ is conjugate to $\tau$ if and only if $\overline{\sigma}=\overline{\tau}.$ \end{lemma}
and 
\begin{lemma}
\label{commutator}
Let $\sigma \tau$ be conjugate to $\tau$. Then $\sigma$ is a commutator $[\tau, a],$ where 
$\sigma \tau = a \tau a^{-1}.$
\end{lemma}

\section{Cycle decomposition of $\sigma$ is \emph{infinite}}
\label{infinite}
Below, we shall consider a fixed point of $\sigma$ a cycle (of length one). We need the following basic result:

\begin{lemma}
\label{horizontal}
If $\sigma$ has an infinity of cycles, we can label $\Omega = \mathbb{Z}^2$ in such a way that 
\[
\sigma(i, j) \in \mathbb{Z}\times \{j\},
\]
for all $(i, j) \in \mathbb{Z}^2.$
In other words, horizontal lines in the integer lattice are left invariant by $\sigma.$
\end{lemma}
\begin{proof}
First, consider the set $C$ of \emph{infinite} cycles of $\sigma.$ The set $C$ is either  finite, or infinite. If $C$ is finite,  let $C = \{c_1, \dotsc, c_k\},$ and then let $i(c_k) = k.$ For any cycle  $c \in C,$ represent $c$ by
\begin{equation}
\label{infcycleeq}
c(i, j) = \begin{cases}
(i+1, j), \qquad j=i(c),\\
(i, j), \qquad \mbox{otherwise},
\end{cases}
\end{equation}

Now, let $D$ be the set of finite cycles of $\sigma.$ Since $C$ is finite, $D$ is infinite. Establish a bijection $f$ between $D$ and $\mathbb{Z}\times (\mathbb{Z}\backslash \{1, \dotsc, k\}.$ Let
\[
D_i = \{ d\in D | y(f(d)) = i.\}
\]
The elements of each of the $D_i$ can obviously be arranged in a bi-infinite sequence on the line $y=i.$  which, together with the above arrangement of the cycles in $C$ gives us the desired result.  If $C$ is infinite and $D$ is empty, simply let $C=(\dotsc, c_{-k}, \dotsc, c_0, \dotsc, c_k, \dotsc),$ and let $i(c_l) = l,$ as before. If $C$ and $D$ are both infinite, arrange the infinite cycles on the lines with even $i$ and the finite cycles on the lines with odd $i,$ and finally if $C$ is infinite and $D$ is finite, arrange all but one cycles of $C$ but treat $c_0$ specially: namely, put all the cycles in $D$ in sequence on a segment (say, $0< x  \leq N)$ of the $x$-axis, and make $c_0$ act on the $x$-axis by a shift which skips that segment (that is, $c_0(x, 0) = (x+1, 0),$ whenever $x \notin [0, N].$ $c_0(0, 0) = (N+1, 0),$ and $c_0(l, 0) = (l, 0),$ for $l \in \{0, \dots, N\}.$
\end{proof}
Now we are ready to prove the theorem in the case where the cycle decomposition of $\sigma$ is infinite. Let $\tau(i, j) = (i, j+1),$ for all $i, j$ The permutation $\tau$ has an infinite number of infinite cycles.
Consider $\psi=\sigma \tau.$ Each cycle of $\psi$ is infinite (because $y(\psi(p)) > y(p),$ for all $p \in \mathbb{Z}^2$), and $\psi$ has an infinite number of cycles (since $(i, 0)$ and $(k, 0)$ lie in different cycles of $\psi$ whenever $i\neq k.$). Therefore, by Lemma \ref{conjugacy1}, $\sigma \tau = a \tau a^{-1},$ for some $a,$ so $\sigma$ is a commutator. It remains to show that $\tau$ and $a$ generate a transitive group. We can pick $a$ so that $a(i, 0)=(i+1, 0)$ (since the only requirement on $a$ is that it bijectively map cycles of $\tau$ to cycles of $\sigma \tau$), and then clearly $\tau^l a^k(0, 0) = (k, l),$ for any $k, l.$
\begin{remark}
\label{transremark}
An examination of the argument above shows that when one of the cycles of $\sigma$ is infinite, then $\sigma$ and $\tau$ generate a transitive group (put the infinite cycle on the $x$-axis, then the origin can be moved by $\sigma$ to any orbit of $\tau,$ and then moved by $\tau$ to any lattice point). Indeed, we can modifiy the argument slightly to show that $\tau$ and $\sigma$ generate a transitive group even when $\sigma$ has no infinite cycles, but moves an infinite number of points -- place a set of cycles of $\sigma$ in such a way that their left-most points lie on the diagonal $x=y.$ Then, a suitable iteration of $\tau \sigma$ will move the origin to an arbitrary vertical line, and the previous argument works again.
\end{remark}
\section{The cycle decomposition of $\sigma$ is \emph{finite}}
\label{finitecycle}
In this case we prove the following lemma:
\begin{lemma}
\label{finitecyc}
If the cycle decomposition of $\sigma$ is finite, then we can label $\Omega=\mathbb{Z}^2$ in such a way that 
\[
\sigma(i,j) \in \mathbb{Z} x \{j-1,j,j+1\}  \mbox{ for all } (i,j) \in \mathbb{Z}^2.
\]
\end{lemma}
The proof of Lemma \ref{finitecyc} occupies the rest of this section.
The argument  is geometric (just as in Section \ref{infinite}). If the cycle decomposition of $\sigma$ is finite, there are the following possibilities:
\subsection{$\sigma$ has one infinite cycle and some finite cycles} In this case, first assume that there are no finite cycles. obviously the one cycle occupies the entire integer lattice $\mathbb{Z}^2, $ and we make $\sigma$ act as follows:
In the upper half-plane, $\sigma$ sweeps out semi-squares, as follows: 
\begin{multline*}(0, 0) \rightarrow (1, 0) \rightarrow (1, 1) \rightarrow (0, 1) \rightarrow (-1, 1) \rightarrow (-1, 0) \\
\rightarrow  (-2, 0) \rightarrow (-2, 1) \rightarrow (-2, 2) \rightarrow (-1, 2) \dotso
\end{multline*}
The pattern in the lower halfplane $x(p) < 0$ is the pattern in the upper halfplane,  reflected in the line $y=-1/2,$ and with the arrows reversed. Finally, $\sigma(0, -1) = (0, 0).$
If there are some finite cycles, arrange them on a (finite portion) of the $x$-axis, and adapt the pattern in the infinite cycle in the obvious way.
\subsection{$\sigma$ has at least two infinite  cycles $c_1, c_2,$ and possibly some finite cycles}
First, assume there are no finite cycles. Let $c_1, \dots, c_k$ be the infinite cycles of $\sigma.$ Arrange $c_1$ in the upper halfplane,$ c_j,$ for $j\in \{2, \dotsc, {k-1}\}$ on horizontal lines $y(p)=-j+1,$ and $c_k$ as the set $\{p\in \mathbb{Z}^2 \quad | \quad y(p) < 2-k\}$  We will describe the pattern of $c_1.$ Then $c_k$ will be just a reflection (and $c_2, \dotsc, c_{k-1}$ are obvious).

The carrier of $c_1$ is the set $H=\{p\in \mathbb{Z}^2 | y(p) \geq 0\}.$ We divide $H$ into the sets $H_1 =\{p\in H | x(p)\geq 0\},$ and $H_2 = H \backslash H_1.$
On $H_1,$ $\sigma$ acts in a square sweep-out almost as before:
\[
(0,. 0) \rightarrow (1, 0) \rightarrow (1, 1) \rightarrow (0, 1) \rightarrow (0, 2) \rightarrow (1, 2) \rightarrow (2, 2) \rightarrow (2, 1) \rightarrow (2, 0) \rightarrow \dots\]
and on $H_2$ the action is symmetric (but with the arrows reversed). $\sigma(-1, 0) = (0, 0).$

If there are some finite cycles, arrange them on (a finite segment of) the $x$-axis, and adapt the pattern above.

\subsection{The conclusion of the argument}
The point of the previous arrangement of $\sigma$ is that it never moves any point in the integer lattice by more than one unit vertically. Consider now the permutation $\tau,$ given by $\tau(i, j) = (i, j+2).$ The permutation $\tau$ has an infinite number of infinite cycles (just as its namesake in Section \ref{infinite}).
The permutation $\sigma \tau$  likewise has an infinite number of infinite cycles (since $y(\sigma \tau (p)) - y(p)) \geq 1$), so $\sigma \tau = a \tau a^{-1}.$ and so $\sigma$ is a commutator.

To show that $\tau$ and $a$ generate a transitive subgroup of $S_\infty$, it is enough to show that $\tau$ and $\sigma$ generate a transitive subgroup of $S_\infty.$ For this, it suffices to note for any $(i, j) \in \mathbb{Z}^2,$ there is some $k> 0,$ such that $\tau^k(i, j)$ is contained in the cycle $c_1$ of $\sigma,$ and so $\sigma^l \tau^k (i, j) = (0, 1),$ for some integers $l, k.$
\begin{remark}
\label{transremark2}
Again, note that we have actually shown that $\sigma$ and $\tau$ generate a transitive subgroup.
\end{remark}

\section{Products of powers}
\label{powersec}
Here we state the two results needed in the proof of Theorem \ref{sullthminfn}.
\begin{theorem}
\label{powers}
Let $\sigma \in S(\Omega),$ where $\Omega$ is countably infinite. Then 
\[
\sigma=\prod_{i=1}^k\alpha_i^{n_i} \prod_{i=j}^m\beta_j^{l_j}
\]
 for any $k,m>1$  where $\alpha_1, \dotsc, \alpha_k$ are commuting permutations in $S(\Omega),$ as are $\beta_1, \dots, \beta_m.$ If $\sigma$ moves an infinite number of elements of $\Omega,$ then, in addition, the group generated by $\alpha_1, \dotsc, \alpha_k, \beta_1, \dotsc, \beta_m$ acts transitively on $\Omega.$
\end{theorem}
\begin{proof}
First note that the $n$-th power of an infinite cycle $C$ is a product of $n$ disjoint cycles, and so the permutation $C_\infty$  whose cycle decomposition is an infinite set of infinite cycles can be written as
\[
C_\infty = \prod_{i=1}^k \alpha_i^{n_i},
\]
for any $k>0.$  To see this, write $C_\infty$ as a group of $n_1$ cycles, followed by a group of $n_2$ cycles, etc, and then a repeating group of $n_k$ cycles. Since the supports of $\alpha_1, \dotsc, \alpha_k$ are disjoint, the permutations commute.
The statement of the theorem now follows from Remarks \ref{transremark} and \ref{transremark2}.
\end{proof}
\begin{remark}
The number of decompositions described in the statement of Theorem \ref{powers} (with a prescribed set of exponents) is  uncountably infinite.
\end{remark}
We will also need the following result, which was proved in \cite{DrosteGobel} as Lemma 4.2:
\begin{theorem}
\label{dglemma}
let $\sigma \in S(\Omega)$ have no infinite cycles. Then, for any $k>1,$ there are transitive cycles $c_1, c_2$ such that 
\[
\sigma = c_1^k c_2^k.
\]
\end{theorem}
\bibliographystyle{plain}
\bibliography{rivin}
\end{document}